\documentclass[11pt, a4paper]{amsart}
\usepackage{amssymb,amsthm, dsfont, a4}
\usepackage{hyperref, xcolor}
\usepackage{verbatim}   
\usepackage[all]{xy}
\usepackage[active]{srcltx}
\usepackage[short,nodayofweek]{datetime}

\definecolor{my-blue}{cmyk}{1,0.6,0,0}

\definecolor{my-green}{cmyk}{0.8,0,1,0.5}

\hypersetup{
colorlinks=true, 
linkcolor=my-blue, 
citecolor=my-green,
urlcolor=black
}

\newcommand\FF{{\mathbb F}}

\newcommand\GG{{\mathbb G}}
\newcommand\NN{{\mathbb N}}
\newcommand\ZZ{{\mathbb Z}}
\newcommand\QQ{{\mathbb Q}}
\newcommand\CC{{\mathbb C}}

\newcommand{\LL}{\mathbb{L}}
\newcommand{\Fq}{{\FF_{\! q}}}
\newcommand{\Fqtimes}{{\FF_{\! q}^\times}}
\newcommand{\Fpp}{{\FF_\pfrak}}
\newcommand\CI{{\mathbb C}_\infty}
\newcommand{\barK}{\overline{K}}

\newcommand{\cA}{\mathcal{A}}

\newcommand\pfrak{{\mathfrak p}}
\newcommand{\degp}{d_\pfrak} 

\newcommand{\mot}{\mathsf{M}}  

\newcommand{\dumot}{\mathfrak{M}}  

\newcommand{\one}{\mathds{1}}

\newcommand{\bigmid}{\, \Big| \,}

\newcommand{\hd}[2]{\partial_t^{(#1)}\!\!\left(#2\right)}
\newcommand{\hde}[1]{\partial_t^{(#1)}}  
\newcommand{\id}{\mathrm{id}}

\newcommand{\tr}{\mathrm{tr}}
\newcommand{\sep}{\mathrm{sep}}

\newcommand{\tate}{\CC_\infty\cs{t}}  
\newcommand{\laurent}{\CC_\infty(\!(t)\!)}  
\newcommand{\ps}[1]{[\![#1]\!]}  
\newcommand{\ls}[1]{(\!(#1)\!)}
\newcommand{\cs}[1]{\langle #1 \rangle} 

\newcommand{\vect}[1]{\text{\boldmath $#1$\unboldmath}} 
\newcommand{\svect}[2]{\left( \begin{matrix} {#1}_{1}\\ \vdots \\ {#1}_{#2}\end{matrix}\right)}

\newcommand{\isom}{\cong}

\newcommand{\betr}[1]{\lvert #1\rvert}

\newcommand{\longsurjarrow}{\longrightarrow \hspace*{-3.3mm}\rightarrow} 
\newcommand{\partdef}[1]{ \left\{ \begin{array}{ll} #1 \end{array} \right. }

\DeclareMathOperator{\Hom}{Hom}
\DeclareMathOperator{\End}{End}
\DeclareMathOperator{\im}{im}
\DeclareMathOperator{\ch}{char}
\DeclareMathOperator{\Mat}{Mat}
\DeclareMathOperator{\GL}{GL}
\DeclareMathOperator{\Gal}{Gal}
\DeclareMathOperator{\Lie}{Lie}

\DeclareMathOperator{\Ker}{ker}

\def\markdef{\bf }

\theoremstyle{plain}
\newtheorem{thm}{Theorem}[section]
\newtheorem{cor}[thm]{Corollary}

\newtheorem{prop}[thm]{Proposition}

\newtheorem*{theoremA}{Theorem \ref{thm:description-of-motivic-grp}}
\newtheorem*{theoremB}{Theorem \ref{thm:description-of-galois-rep}}

\theoremstyle{definition}
\newtheorem{defn}[thm]{Definition}
\newtheorem{exmp}[thm]{Example}
\newtheorem{rem}[thm]{Remark}

\begin{document}

\title[thin t-adic Galois representations]{Anderson t-modules with thin t-adic Galois representations}
\author{Andreas Maurischat}
\address{\rm {\bf Andreas Maurischat}, Lehrstuhl A f\"ur Mathematik, RWTH Aachen University, Germany }
\email{\sf andreas.maurischat@matha.rwth-aachen.de}

\newdate{date}{04}{02}{2021}
\date{\displaydate{date}}




\begin{abstract}
Pink has given a qualitative answer to the Mumford-Tate conjecture for Drinfeld modules in the 90s. He showed that the image of the $\pfrak$-adic Galois representation is $\pfrak$-adically open in the motivic Galois group for any prime $\pfrak$.
In contrast to this result, we provide a family of uniformizable Anderson $t$-modules for which the Galois representations of their $t$-adic Tate-modules are ``far from'' having $t$-adically open image in their motivic Galois groups. Nevertheless, the image is still Zariski-dense in the motivic Galois group which is in accordance to the Mumford-Tate conjecture. For the proof, we explicitly determine the motivic Galois group as well as the Galois representation for these $t$-modules.
\end{abstract}

\maketitle

\setcounter{tocdepth}{1}
\tableofcontents


\section{Introduction}\label{sec:intro}

Given an abelian variety $\mathcal{A}$ of dimension $d$ over a number field $F$, there are several groups attached to it. Among others, there is the associated Hodge group (also called Mumford-Tate group) which is an algebraic subgroup $G_{\mathcal{A}}$ of $\GL(V)\cong \GL_{2d,\QQ}$ where $V:=H_1(\cA(\CC),\QQ)$ denotes the singular homology of $\cA(\CC)$ with respect to some embedding $F\to \CC$. Furthermore,
the rational $\ell$-adic Tate module $V_\ell(\cA)=\left(\varprojlim_{n} \cA[\ell^{n+1}]\right)\otimes_{\ZZ_\ell} \QQ_\ell$ of $\cA$, is naturally isomorphic to $V\otimes_{\QQ} \QQ_\ell$, and carries an action of the absolute Galois group $\Gal(F^{\sep}/F)$ via the action on the torsion points, i.e.~one has a group  homomorphism
\[ \varrho_\ell: \Gal(F^{\sep}/F) \longrightarrow  \GL_{2d}(\QQ_\ell). \]
It is well-known that the image of $\varrho_\ell$ is a compact subgroup of the $\ell$-adic points
 $G_{\mathcal{A}}(\QQ_\ell)$ of $G_{\mathcal{A}}$ (compact in the $\ell$-adic topology).
The Mumford-Tate conjecture asks whether the image is even Zariski-dense.
As the Mumford-Tate group  $G_{\mathcal{A}}$ is always a reductive group over a field of characteristic zero, an equivalent  formulation is whether there is an ($\ell$-adically) open subgroup of $\im(\varrho_\ell)$ which is also an ($\ell$-adically) open subgroup of  $G_{\mathcal{A}}(\QQ_\ell)$.

The same questions arise for Drinfeld modules and Anderson $t$-modules -- the function field analogues of elliptic curves and abelian varieties. In \cite{rp:mtcdm}, Pink showed that the Mumford-Tate conjecture is true for Drinfeld modules by showing the $\pfrak$-adic openness of the image of the Galois representation.
For $\pfrak=t$, this statement has been used by Chang and Papanikolas (see \cite[\S 3.5]{cc-mp:aipldm}) to determine the motivic Galois group of Drinfeld modules using the explicit connection via Anderson generating functions between elements in the Tate module and the rigid analytic trivialization.

In the function field setting, however, reductive groups can have non-open compact $\pfrak$-adic subgroups\footnote{I thank Urs Hartl for pointing that out to me.}. Even more, there exist Anderson $t$-modules whose Mumford-Tate groups are not reductive. For a general treatment of compact subgroups of linear algebraic groups, we refer to Pink's article \cite{rp:cslag}.
 
 
The content of our article is to show that indeed the image of the Galois representation might be not $\pfrak$-adically open and can even have a ``thin image'' (see Def.~\ref{def:thin-image} for the definition). We restrict to the case $\pfrak=t$, but are confident that also the images for other $\pfrak$ can be thin. 
For the proof, we will use the explicit connection between the $t$-adic Tate module and the rigid analytic trivialization of the motive given in \cite{am:ptmsv} which is a generalization of the corresponding one in \cite{cc-mp:aipldm} mentioned above.
A similar approach for other primes $\pfrak$ needs more theoretical background. This is in the scope of a future project.

\medskip

We will show the theorem on thin images by examining an explicit family of $t$-modules where this happens.
The family that we consider are the prolongations $\rho_k C$ of the Carlitz module $C$ which were introduced in \cite{am:ptmaip}. The main results are the computations of the motivic Galois group and the $t$-adic Galois representation for these $t$-modules from which the result on the thin image (Corollary \ref{cor:thin-image}) and the Zariski-density of the image (Corollary \ref{cor:image-Zariski-dense}) are deduced.
The main theorems are (for notation see Section \ref{sec:notation})

\begin{theoremA}
The motivic Galois group $\Gamma_{\rho_k C}$ of the $t$-module $\rho_k C$ is given by
\[   \Gamma_{\rho_k C}(R) = \left\{  \begin{pmatrix} a_0 & a_1 & \cdots & a_k \\ 0 & a_0 & \ddots & \vdots \\ \vdots & \ddots & \ddots &a_1  \\ 0 & \cdots & 0 & a_0 \end{pmatrix} \,\,\middle|\,\, 
 a_0\in R^\times, a_1,\ldots, a_k\in R
\right\} \]
for all $\Fq(t)$-algebras $R$.
\end{theoremA}

\begin{theoremB}
The Galois representation on the $t$-adic Tate module of $\rho_kC$ is given by
\[ \begin{array}{rccccl}
 \varrho_t: &\Gal(K^{\sep}/K) &\longsurjarrow  &\Fq\ps{t}^\times & \longrightarrow &\GL_{k+1}(\Fq\ps{t}),\\ 
 & \gamma &\longmapsto & a:=a_\gamma & \longmapsto & \rho_{[k]}(a)=
\left( \begin{matrix} a   & \hd{1}{a}  & \cdots &\hd{k}{a}  \\ 0 & a  & \ddots & \vdots \\ \vdots & \ddots & \ddots & \hd{1}{a}   \\ 0 & \cdots & 0 & a  \end{matrix} \right).
\end{array} \]
Here the first map is the usual Galois representation on the $t$-adic Tate module of the Carlitz module, and $\hd{j}{a}$ is the $j$-th hyperderivative of $a\in \Fq\ps{t}$.
\end{theoremB}

\begin{rem}
We expect that one can generalize the result on the thin image to prolongations of any abelian uniformizable $t$-modules, but we will not pursue  it in this paper.
\end{rem}

\begin{rem}
Recently, M.~Frantzen \cite{mf:novgram} discovered other examples of $t$-modules for which the Galois image is not $t$-adically open in the motivic Galois group. These examples have rank $1$ and dimension $d$ divisible by the characteristic $p$, i.e. they are $t$-modules which become isomorphic over $\CI$ to
the tensor powers $C^{\otimes d}$ of the Carlitz module (see Thm.~5.2 ibid.). We compute the density for these Galois representations in Theorem \ref{thm:tensor-powers}.
\end{rem}

\section{Notation}\label{sec:notation}

\subsection{Rings and operators}\label{subsec:rings-ops}

Let $K=\Fq(\theta)$ be the rational function field over the finite field $\Fq$ with $q$ elements and characteristic $p$. Its completion with respect to the absolute value determined by $\betr{\theta}=q$ is the Laurent series ring $K_\infty:=\Fq\ls{1/\theta}$, and we let $\CC_\infty=\widehat{\overline{K_\infty}}$ be the completion  of the algebraic closure $\overline{K_\infty}$ of $K_\infty$. We also fix an embedding of the algebraic closure $\barK$ of $K$ into $\overline{K_\infty}$.

On $\CI$ and all its subfields, we have the $q$-power Frobenius endomorphism $\tau:\CI\to \CI,x\mapsto x^q$ which will be extended $t$-linearly to the polynomial ring $\CI[t]$ (still denoted by $\tau$). The inverse of the automorphism $\tau$ will be denoted by $\sigma$.

We consider the skew-polynomial ring $K\{\tau\}=\{ \sum_{k=0}^n x_k\tau^k\mid n\geq 0, x_k\in K\}$ of $\Fq$-linear polynomial maps with multiplication rule $\tau\cdot x=x^q\cdot \tau$ induced from composition of maps, as well as the skew-polynomial ring $\barK\{\sigma\}=
\{ \sum_{k=0}^n x_k\sigma^k\mid n\geq 0, x_k\in \barK\}$ with multiplication rule 
$\sigma\cdot x=x^{1/q}\cdot \sigma$. For later use, we already mention that $\barK\{\sigma\}$ is the opposite ring to 
$\barK\{\tau\}$ via the anti-isomorphism $\sum_{k=0}^n x_k\tau^k\mapsto \sum_{k=0}^n \sigma^k x_k= \sum_{k=0}^n (x_k)^{1/q^k}\cdot \sigma^k$.

We denote by $\CI\cs{t}$ the Tate algebra over $\CI$. This is the Gauss completion of $\CI[t]$, i.e.~the algebra of power series $\sum_{n=0}^\infty x_nt^n\in \CI\ps{t}$ such that $\lim_{n\to \infty} \betr{x_n}=0$. Its field of fractions will be denoted by $\LL$. The automorphisms $\tau$ and $\sigma$ will be extended continuously to $\tate$ (and further to $\LL$), i.e.
\[\tau\left( \sum_{n=0}^\infty x_nt^n\right)=\sum_{n=0}^\infty x_n^qt^n\quad \text{and}\quad
\sigma\left( \sum_{n=0}^\infty x_nt^n\right)=\sum_{n=0}^\infty (x_n)^{1/q}t^n. \]

\medskip 

On the Laurent series ring $\laurent$ (in which all the rings above embed) we consider the hyperdifferential operators with respect to $t$, i.e.~the sequence of $\CI$-linear maps $(\hde{n})_{n\geq 0}$ given by
\[    \hd{n}{\sum_{i=i_0}^\infty x_it^i} =\sum_{i=i_0}^\infty \binom{i}{n} x_it^{i-n}. \]
where $\binom{i}{n}\in \FF_{\! p}\subset \Fq$ is the binomial coefficient modulo $p$.\footnote{The operator $\hde{n}$ is the positive characteristic analogue of $\frac{1}{n!}\left(\frac{d}{dt}\right)^n$.}
The main properties of the hyperdifferential operators are the generalized Leibniz rule
\begin{equation}\label{eq:leibniz-rule}
\hd{n}{f\cdot g}=\sum_{i=0}^n \hd{i}{f}\cdot \hd{n-i}{g}
\end{equation}
for all $n\geq 0$, and all $f,g\in \laurent$, as well as 
the iteration rule
\begin{equation} \label{eq:iteration-rule}
\hde{n}\circ \hde{m}=\binom{n+m}{n} \hde{n+m}
\end{equation}
for all $n,m\geq 0$.

If $f$ is a power series $f(t)\in \CI\ps{t}$, one has the familiar Taylor series identity
\begin{equation}\label{eq:taylor-identity}
 f(t)= \sum_{i\geq 0} (\hd{i}{f})(0) t^i, 
\end{equation}
where $(\hd{i}{f})(0)$ is the value of the hyperderivative  $\hd{i}{f}$ at $0$.

It is obvious that the hyperdifferential operators commute with the twistings $\tau$ and~$\sigma$.

For more background on hyperdifferential operators see for example \cite[\S 27]{hm:crt} (called iterative higher derivations there).

\medskip

From the hyperdifferential operators, one gets a family (in $k \geq 0$) of 
maps
 $\rho_{[k]} : \laurent \rightarrow \Mat_{(k+1) \times (k+1)}(\laurent)$ sending $f\in \laurent$ to
\begin{equation}\label{eq:rho_n}
 \rho_{[k]}(f) := \begin{pmatrix} f & \hd{1}{f} & \cdots & \hd{k}{f} \\ 0 & f & \ddots & \vdots \\ \vdots & \ddots & \ddots & \hd{1}{f} \\ 0 & \cdots & 0 & f \end{pmatrix}.
\end{equation}
A crucial point is that these maps are even homomorphisms of $\CI$-algebras (see \cite[\S 2.2]{am-rp:iddbcppte}
or \cite[\S 2.1]{am:ptmaip}).

\subsection{$t$-modules, $t$-motives and dual $t$-motives}\label{subsec:t-modules-et-al}

For $t$-modules, $t$-motives and dual $t$-motives, we will use the notation as in \cite{am:ptmaip} and  \cite{db-mp:ridmtt}.

A {\markdef $t$-module} $(E,\Phi)$ over $K$ (or shortly, $E$) consists of an algebraic group $E$  over $K$ with $\Fq$-action which is isomorphic to $\GG_a^d$ for some $d>0$ (also called $\Fq$-module scheme), and an $\Fq$-algebra homomorphism
\[ \Phi:\Fq[t]\to \End_{{\rm grp},\Fq}(E)\isom \Mat_{d\times d}(K\{\tau\}),t\mapsto \Phi_t, \]
with the additional property that $\Phi_t-\theta\cdot \id_E$ induces a nilpotent endomorphism on $\Lie(E)$.

To a $t$-module $(E,\Phi)$ one associates a {\markdef $t$-motive}. This is the left $K[t]\{\tau\}$-module 
$\mot(E):=\Hom_{{\rm grp},\Fq}(E,\GG_a)$ with $t$-action given by composition with 
$\Phi_t\in \End_{{\rm grp},\Fq}(E)$ and left-$K\{\tau\}$-action given by composition
with elements in $K\{\tau\}\isom \End_{{\rm grp},\Fq}(\GG_a)$. The $t$-motive $\mot(E)$ and the $t$-module $E$ are called {\markdef abelian}, if the $t$-motive $\mot(E)$ is finitely generated as $K[t]$-module in which case it is even free as $K[t]$-module.\footnote{Be aware that in other sources abelianess is often part of the definition of a $t$-motive, e.g.~in Anderson's original definition in \cite{ga:tm}. The last statement on the freeness is also proven ibid.}

Similarly, the {\markdef dual $t$-motive} associated to a $t$-module $(E,\Phi)$ is the left $\barK[t]\{\sigma\}$-module
$\dumot(E):=\Hom_{{\rm grp},\Fq}(\GG_{a,\barK},E_{\barK})$ with $t$-action given by composition with 
$\Phi_t\in \End_{{\rm grp},\Fq}(E)$ and left-$\barK\{\sigma\}$-action given by right composition
with elements in $\barK\{\tau\}\isom \barK\{\sigma\}^{\rm op}$. 
The dual $t$-motive $\dumot(E)$ and the $t$-module $E$ are called {\markdef $t$-finite} if $\dumot(E)$ is finitely generated as $\barK[t]$-module in which case it is even free as $\barK[t]$-module.

For an abelian $t$-motive $\mot(E)$, let $\vect{m}=(m_1,\ldots,m_r)^\tr$ be a $K[t]$-basis (written as a column vector). Then a {\markdef rigid analytic trivialization} (if it exists) is a matrix $\Upsilon\in \GL_{r}(\tate)$
such that $\tau(\Upsilon\vect{m})= \Upsilon \vect{m}$.

Similarly, for a $t$-finite dual $t$-motive $\dumot(E)$, let $\vect{e}=(e_1,\ldots,e_r)^\tr$ be a $\barK[t]$-basis. Then a {\markdef rigid analytic trivialization} (if it exists) is a matrix $\Psi\in \GL_{r}(\tate)$
such that $\sigma(\Psi^{-1}\vect{e})= \Psi^{-1} \vect{e}$.

The $t$-modules that we consider in this paper are both abelian and $t$-finite. Furthermore, their $t$-motives resp.~dual $t$-motives are both rigid analytically trivial which is equivalent to the $t$-module being \emph{uniformizable}.

\begin{exmp}\label{exmp:carlitz}
The Carlitz module $(C,\phi)$ over $K$ is the additive group scheme $\GG_{a,K}$ over $K$ with $\Fq[t]$-action \[ \phi:\Fq[t]\to \End_{{\rm grp},\Fq}(\GG_{a,K})\isom K\{\tau\}\]
given by $\phi_t=\theta+\tau$.

Its corresponding $t$-motive $\mot(C)$ is given as $\mot(C)=K[t]e$ with $\tau$-action determined by 
$\tau(e)=(t-\theta)e$, and its rigid analytic trivialization is $\Upsilon=\omega^{-1}$ where $\omega\in \CI\cs{t}$ is the Anderson-Thakur function
\[ \omega(t)=\sqrt[q-1]{-\theta}\cdot \prod_{i=0}^\infty \left(1-\frac{t}{\theta^{q^i}}\right)^{-1} \]
for a fixed choice of $(q-1)$-st root $\sqrt[q-1]{-\theta}$ of $-\theta$ (see \cite[Proof of Lemma 2.5.4]{ga-dt:tpcmzv}).

 The dual $t$-motive $\dumot(C)$ is given as $\dumot(C)=\barK[t]e$ with $\sigma$-action determined by $\sigma(e)=(t-\theta)e$, and its rigid analytic trivialization is $\Psi=(t-\theta)^{-1}\omega^{-1}$.
 
In both cases, we chose $e=\id_{\GG_{a}}$ as basis vector.
\end{exmp}

The $t$-modules that we are mainly dealing with in this paper are the \emph{prolongations} of the Carlitz module as defined in \cite[Sect.~6]{am:ptmaip} which we recall in Section \ref{sec:family}.

\subsection{The Tate module}\label{subsec:tate-module}

The $t$-adic Tate module of a $t$-module $(E,\Phi)$ is the $\Fq\ps{t}$-module
\[ T_t(E):=\varprojlim_{n} E[t^{n+1}],\quad \text{where }E[t^{n+1}]=\{e\in E(\CI)\mid 
\Phi_{t^{n+1}}(e)=0\} \]
denote the $t^{n+1}$-torsion points of $E$. The torsion points even lie in $E(K^{\sep})$ (as $
\Phi_{t^{n+1}}$ is a finite morphism), and the Tate module $T_t(E)$ is equipped with a continuous action of the absolute Galois group $\Gal(K^{\sep}/K)$. After a choice of $\Fq\ps{t}$-basis of $T_t(E)$, this leads to a continuous representation $\varrho_t:\Gal(K^{\sep}/K)\to \GL_r(\Fq\ps{t})$ where
$r={\rm rk}_{\Fq\ps{t}}(T_t(E))$ is the rank of the Tate module.

In \cite{am:ptmsv}, we proved an explicit isomorphism between the Tate module $T_t(E)$ and the $\tau$-invariants
$\Hom_{K[t]}^\tau(\mot(E),\CI\ps{t})$ of the $\CI\ps{t}$-dual of the $t$-motive $\mot(E)$. 
In this article, however, we only need the description of the Tate module coming from this isomorphism in the case where $\mot(E)$ is abelian and rigid analytically trivial. We explain this description in the following.

Let $E(\CI)\ps{t}$ be the formal power series in $t$ with coefficients in the $\CI$-points of $E$, and the $\Fq\ps{t}$-action by formal multiplication using the $\Fq$-action on the $\Fq$-module scheme $E$, i.e.~
\[ \left(\sum_{i\geq 0} c_it^i\right)\cdot \left(\sum_{n\geq 0} e_nt^n\right) = \sum_{m\geq 0}\left(\sum_{i+n=m} c_i\cdot e_n\right)t^m \]
for $\sum_{i\geq 0} c_it^i\in \Fq\ps{t}$ and $\sum_{n\geq 0} e_nt^n\in E(\CI)\ps{t}$.
Then $T_t(E)$ is isomorphic as $\Fq\ps{t}$-module to
the $\Fq\ps{t}$-subspace
\[ \hat{H}_E:=\left\{ \sum_{n\geq 0} e_nt^n \in E(\CI)\ps{t} \,\middle|\,
\sum_{n\geq 0} \Phi_t(e_n)t^n = \sum_{n\geq 0} e_nt^{n+1} \right\} \]
 by sending a compatible system $(e_n)_{n\geq 0}$ of $t^{n+1}$-torsion points to the series 
$\sum_{n\geq 0} e_nt^n$  (see \cite[Prop.~3.2]{am:ptmsv}).

Furthermore, assume that $\mot(E)$ is abelian and rigid analytically trivial.
%
%
%
Let $m_1,\ldots, m_r\in \mot(E)$ be a $K[t]$-basis of $\mot(E)$, and $\kappa_1,\ldots, \kappa_d\in \mot(E)$ a $K\{\tau\}$-basis of $\mot(E)$. Hence, $\kappa_1,\ldots, \kappa_d$ induce a \emph{choice of coordinate system}, i.e.~an isomorphism of $\Fq$-module schemes
\[ \kappa:E\to \GG_a^d,e\mapsto \left(\begin{smallmatrix}
\kappa_1(e) \\ \vdots \\ \kappa_d(e) \end{smallmatrix}\right). \]
Let $S\in \Mat_{d\times r}(K[t])$ be such that
  \[  \svect{\kappa}{d} = S\cdot \svect{m}{r},\]
  which exists, since $\{m_1,\ldots, m_r \}$ is a $K[t]$-basis.
 
Further, let $\Upsilon\in \GL_r(\CI\cs{t})\subseteq  \GL_r(\CI\ps{t})$ be a rigid analytic trivialization for $\mot(E)$ with respect to the basis $\{m_1,\ldots, m_r\}$.
Then we have the following.
\begin{prop}[{cf.~\cite[Prop.~5.3]{am:ptmsv}}] \label{prop:basis-of-h-hat}
With the previous notation, an $\Fq\ps{t}$-basis of $\hat{H}_E\subseteq E(\CI)\ps{t}\cong (\CI\ps{t})^d$ is given 
by the columns of $ S\cdot \Upsilon^{-1}$ with respect to the coordinate system $\kappa$.
\end{prop}

\section{Density of Galois representations}\label{sec:density}

In this section, we introduce the notion of \emph{density} of Galois representations, and define when we consider the image to be \emph{thin}.

Let $\pfrak\lhd \Fq[t]$ be a prime ideal, and $A_\pfrak$ be the $\pfrak$-adic completion of $A:=\Fq[t]$. Let $\Fpp:= \Fq[t]/\pfrak$ be its residue field, and $\degp:=[\Fpp:\Fq]$ the residue degree.

Throughout this section, let $\varrho:\Gal(K^\sep/K)\to \GL_r(A_\pfrak)$ be a continuous $\pfrak$-adic Galois representation, and $\Gamma(A_\pfrak):=\overline{\im(\varrho)}$ be the  Zariski-closure of $\im(\varrho)$ in $\GL_r(A_\pfrak)$.

\begin{defn}\label{def:thin-image}
For $N\geq 1$, let $\varrho_{N}:\Gal(K^{\sep}/K)\to \GL_r(A_\pfrak)\to  \GL_r(A_\pfrak/(\pfrak^N))$ be the composition of $\varrho$ with the reduction modulo $\pfrak^N$, and let $D(N)$ be the order of the image $\im(\varrho_{N})$ (which obviously is a finite group).
Further, let $\dim(\Gamma)$ be the dimension of $\Gamma$ as a linear algebraic group.

We define the {\markdef density} $\delta(\varrho)$ of $\varrho$ (in $\Gamma$) as
\[  \delta(\varrho):= \limsup_{N\to \infty} \frac{\log_q(D(N))}{N\cdot \degp \cdot \dim(\Gamma)}. \]

We say that the Galois representation $\varrho$ has a {\markdef thin image} if $\delta(\varrho)<1$.
\end{defn}

In the following proposition, we show that this naming makes sense.

\begin{prop}
\begin{enumerate}
\item The density $\delta(\varrho)$ is bounded by $1$.
\item If the image $\im(\varrho)$ is $\pfrak$-adically open in $\Gamma(A_\pfrak)$, then
$\delta(\varrho)=1$. In particular, a thin image is never open.
\end{enumerate}
\end{prop}

\begin{proof}
For $n\geq 1$, let $G_n:=\Ker\bigl(\Gamma(A_\pfrak)\to \Gamma(A_\pfrak/\pfrak^n)\bigr)$ be the kernel of the reduction map. Then by \cite[Prop.~6.3(a)]{rp:cslag}, 
$G_n/G_{n+1}\isom \Lie(\Gamma)\otimes_{A_\pfrak}\left(\pfrak^n/\pfrak^{n+1}\right)$ is an $\Fpp$-vector space of dimension $\dim_{A_\pfrak}\Lie(\Gamma)=\dim(\Gamma)$.
Hence, for arbitrary $N>n$ we have
\[  \#(G_n/G_N)=q^{(N-n)\cdot \degp\cdot \dim(\Gamma)} \]
as well as
\[ \# (\Gamma(A_\pfrak)/G_N)=\left( \# \Gamma(\Fpp)\right) \cdot q^{(N-1)\cdot \degp\cdot \dim(\Gamma)} \]
Therefore, we obtain
\begin{eqnarray*}
\delta(\varrho) &=&  \limsup_{N\to \infty} \frac{\log_q(D(N))}{N\cdot \degp \cdot \dim(\Gamma)} \\
&\leq &   \limsup_{N\to \infty} \frac{\log_q(\# \Gamma(\Fpp)) + (N-1)\cdot \degp\cdot \dim(\Gamma)}{N\cdot \degp \cdot \dim(\Gamma)} =1
\end{eqnarray*}   
If $\im(\varrho)$ is open in $\Gamma$, then there is some $n\in \NN$ such that $\im(\varrho)\supset G_n$, and we obtain for all $N>n$:
\[  D(N)\geq \# (G_n/G_N)=q^{(N-n)\cdot \degp\cdot \dim(\Gamma)}. \]
Therefore,
\[  \delta(\varrho)= \limsup_{N\to \infty} \frac{\log_q(D(N))}{N\cdot \degp\cdot \dim(\Gamma)} \geq \lim_{N\to \infty} \frac{(N-n)}{N}=1. \qedhere \]
\end{proof}

\section{Prolongations of the Carlitz module}\label{sec:family}

In this section, we recall the \emph{prolongations} of the Carlitz module as defined in \cite[Sect.~6]{am:ptmaip}, as well as the explicit descriptions. Those we will need for computing the motivic Galois group and the $t$-adic Galois representation of their Tate module, and proving that the image of the representation is thin.

\medskip

For $k\geq 1$,  the {\markdef $k$-th prolongation} of the Carlitz module $(C,\phi)$, denoted by $(\rho_k C,\rho_k\phi)$, is the algebraic group $\GG_a^{k+1}$ with $t$-action given by 
\[ (\rho_k\phi)_t = 
 \begin{pmatrix}
\theta & -1 & 0 &  \cdots & 0 \\
0 & \theta &  \ddots &  \ddots &  \vdots \\
\vdots &  \ddots& \ddots & \ddots & 0 \\
\vdots & &  \ddots& \ddots & -1\\
0 & \cdots& \cdots & 0&  \theta 
\end{pmatrix}+ \one_{k+1}\cdot \tau, \]
where $\one_{k+1}$ is the identity matrix in $\GL_{k+1}(K)$.
Compared to the presentation in \cite{am:ptmaip}, we use the reversed basis here.

Let $\{m_0,\ldots, m_k\}$ be the corresponding coordinate functions which not only form a $K\{\tau\}$-basis, but also a $K[t]$-basis of the associated $t$-motive $\mot(\rho_k C)$. The $\tau$-action is given by
\[  \tau\begin{pmatrix}
m_0\\ m_1\\ \vdots \\ m_k \end{pmatrix} = 
\begin{pmatrix}
t-\theta & 1 & 0 &  \cdots & 0 \\
0 & t-\theta &  \ddots &  \ddots &  \vdots \\
\vdots &  \ddots& \ddots & \ddots & 0 \\
\vdots & &  \ddots& \ddots & 1\\
0 & \cdots& \cdots & 0&  t-\theta 
\end{pmatrix} \cdot   \begin{pmatrix}
m_0\\ m_1\\ \vdots \\ m_k \end{pmatrix} =  \rho_{[k]}(t-\theta) \cdot  \begin{pmatrix}
m_0\\ m_1\\ \vdots \\ m_k \end{pmatrix}
.\]
As it is common, we use a column vector with entries in the motive here, as e.g.~Papanikolas \cite{mp:tdadmaicl} does. This is different to the notation in \cite{am:ptmaip}, but compared to that article, we also chose the reversed basis. In total, we ``accidentally'' obtain the same matrix as given there.

A rigid analytic trivialization for $\mot(\rho_k C)$ with respect to this basis is given by
\[   \Upsilon = \rho_{[k]}(\omega^{-1})=  \begin{pmatrix} \omega & \hd{1}{\omega} & \cdots & \hd{k}{\omega} \\ 0 & \omega & \ddots & \vdots \\ \vdots & \ddots & \ddots & \hd{1}{\omega} \\ 0 & \cdots & 0 & \omega \end{pmatrix}^{-1} \in \Mat_{(k+1)\times (k+1)}(\tate)
, \] 
where $\omega(t)\in \tate$ is the Anderson-Thakur function (comp.~Example \ref{exmp:carlitz}), 
and $\hd{j}{\omega}$ is the $j$-th hyperderivative of $\omega$ as defined in Section \ref{subsec:rings-ops}.

\medskip

Let's now consider the subspace $\hat{H}_{\rho_k C}\subset \rho_k C(\CI)\ps{t}$ defined in Section \ref{subsec:tate-module}. Since the given $K[t]$-basis $\{m_0,\ldots, m_k\}$ is also a $K\{\tau\}$-basis, the base change matrix $S$ in Proposition~\ref{prop:basis-of-h-hat}  is the identity matrix, and a basis for $\hat{H}_{\rho_k C}$ is therefore given by the columns of the matrix
\[  \Upsilon^{-1}= \rho_{[k]}(\omega)=  \begin{pmatrix} \omega & \hd{1}{\omega} & \cdots & \hd{k}{\omega} \\ 0 & \omega & \ddots & \vdots \\ \vdots & \ddots & \ddots & \hd{1}{\omega} \\ 0 & \cdots & 0 & \omega \end{pmatrix}. \]

\section{The Motivic Galois group}\label{sec:motivic-grp}

As in the previous section, let $\Upsilon=\rho_{[k]}(\omega)^{-1}$ be the rigid analytic trivialization of $\mot(\rho_k C)$. By \cite[Prop.~6.2]{am:ptmaip}, the matrix $\Psi:=\rho_{[k]}(t-\theta)^{-1}\Upsilon$ is a rigid analytic trivialization of the dual $t$-motive $\dumot(\rho_k C)$ associated to $\rho_k C$.

In the description of its motivic Galois group $\Gamma_{\rho_k C}$, we follow \cite[\S 4.2]{mp:tdadmaicl}:\\
Let $L:=(\Psi_{ij}\otimes 1)_{ij}\in \GL_{k+1}(\LL\otimes_{\barK(t)} \LL)$, and $R:=(1\otimes \Psi_{ij})_{ij}\in \GL_{k+1}(\LL\otimes_{\barK(t)} \LL)$, as well as $\tilde{\Psi}:=L^{-1}R$. Explicitly,
\[ \tilde{\Psi}= \begin{pmatrix} \omega\otimes \omega^{-1} & \hd{1}{\omega}\otimes \omega^{-1}+\omega\otimes \hd{1}{\omega^{-1}} & \cdots & \hd{k}{\omega}\otimes \omega^{-1}+\ldots + \omega\otimes \hd{k}{\omega^{-1}} \\ 
0 & \omega\otimes \omega^{-1} & \ddots & \vdots \\ \vdots & \ddots & \ddots &\hd{1}{\omega}\otimes \omega^{-1}+\omega\otimes \hd{1}{\omega^{-1}} \\ 0 & \cdots & 0 & \omega\otimes \omega^{-1} \end{pmatrix},
\]
as the terms coming from $\rho_{[k]}(t-\theta)^{-1}$ cancel out.

Further, let 
\[ \Fq(t)[X,\det(X)^{-1}]:= \Fq(t)[X_{0,0},\ldots, X_{0,k},X_{1,0},\ldots, X_{k,k},d]/(d\cdot \det(X)-1)\] be the coordinate ring of $\GL_{k+1}$ over $\Fq(t)$, i.e.~the localization of the polynomial ring in $(k+1)^2$ variables, and
\[ \mu:\Fq(t)[X,\det(X)^{-1}]\to \LL\otimes_{\barK(t)} \LL, X_{ij}\mapsto \tilde{\Psi}_{ij}\]
be the homomorphism sending the matrix of indeterminates $X$ to the matrix $\tilde{\Psi}$ entry-wise. Then, $\Gamma_{\rho_k C}$ is the Zariski-closed subgroup (defined over $\Fq(t)$) given by the kernel of $\mu$, i.e.~by the algebraic relations over $K(t)$ satisfied by the entries of $\tilde{\Psi}$.

%
%

By \cite[Thm.~7.2]{am:ptmaip}), $\omega$ is hypertranscendental, i.e.~$\omega$ and all its hyperderivatives $\hd{j}{\omega}$ are algebraically independent over $K(t)$. Hence, there are only the obvious relations of certain entries being equal, and we get the following description of the motivic Galois group.

\begin{thm}\label{thm:description-of-motivic-grp}
The motivic Galois group $\Gamma_{\rho_k C}$ of the $t$-module $\rho_k C$ is given by
\[   \Gamma_{\rho_k C}(R) = \left\{  \begin{pmatrix} a_0 & a_1 & \cdots & a_k \\ 0 & a_0 & \ddots & \vdots \\ \vdots & \ddots & \ddots &a_1  \\ 0 & \cdots & 0 & a_0 \end{pmatrix} \,\,\middle|\,\, 
 a_0\in R^\times, a_1,\ldots, a_k\in R  \right\} \]
for all $\Fq(t)$-algebras $R$.
\end{thm}

\section{The Galois representation}\label{sec:galois-representation}

As mentioned in Section \ref{subsec:tate-module}, the isomorphism between the Tate module $T_t(E)$ and $\hat{H}_E$ is given by turning a compatible system $(e_i)_{i\geq 0}$ of $t^{i+1}$-torsion points into the series $\sum_{i\geq 0}e_it^i$, and vice versa.

The action of the Galois group $\Gal(K^{\sep}/K)$ on the Tate module is therefore the usual Galois action on the coefficients of the series.

\begin{prop}\label{prop:generators-of-torsion-extension}
The  extension of $K$ generated by the torsion points $\rho_k C[t^{n+1}]$ is generated by the set
\[ \left\{  \binom{i+j}{j}(\hd{i+j}{\omega})(0) \bigmid i=0,\ldots, n; j=0,\ldots, k \right\}. \]
\end{prop}

\begin{proof}
By Section \ref{sec:family}, the columns of 
$\rho_{[k]}(\omega)=  \begin{pmatrix} \omega & \hd{1}{\omega} & \cdots & \hd{k}{\omega} \\ 0 & \omega & \ddots & \vdots \\ \vdots & \ddots & \ddots & \hd{1}{\omega} \\ 0 & \cdots & 0 & \omega \end{pmatrix}
$ provide a basis for $\hat{H}_{\rho_k C}$, and hence the coefficients of $t^0, \ldots, t^n$ in the series $\omega, \hd{1}{\omega}, \ldots, \hd{k}{\omega}$ generate the extension
$K(\rho_k C[t^{n+1}])/K$.
As explained in Sect.~\ref{subsec:rings-ops}, 
any series $g(t)\in \tate$ can be written as
\[  g(t)= \sum_{i\geq 0} (\hd{i}{g})(0) t^i \]
(see Equation \eqref{eq:taylor-identity}),
and therefore for all $j\geq 0$,
\[  \hd{j}{\omega} = \sum_{i\geq 0} (\hd{i}{\hd{j}{\omega}})(0) t^i
=\sum_{i\geq 0} \binom{i+j}{j}(\hd{i+j}{\omega})(0) t^i,
\]
where we used the iteration rule \eqref{eq:iteration-rule}.
\end{proof}

%

\begin{cor}\label{cor:same-extensions}
The extension of $K$ generated by the $t^{n+1}$-torsion points of $\rho_k C$ equals the
extension generated by the $t^{m+1}$-torsion points of the Carlitz module for some $n\leq m\leq n+k$.
\end{cor}

\begin{proof}
By the previous proposition, the extension $K(\rho_k C[t^{n+1}])/K$ is generated by the set
\[ \left\{  \binom{i+j}{j}(\hd{i+j}{\omega})(0) \bigmid i=0,\ldots, n; j=0,\ldots, k \right\}. \]
As the Carlitz $t^{m+1}$-torsion extension is generated by 
\[ \{ (\hd{l}{\omega})(0) \mid  l=0,\ldots, m \},\footnote{This is the case $k=0$. But it is already clear from the original definition of the Anderson-Thakur function as the generating series for the $t$-power torsion taking into account the
Taylor series identity \eqref{eq:taylor-identity}.} \]
it remains to prove that there is some $m$ with $n\leq m\leq n+k$ such that
\begin{enumerate}
\item for $l\leq m$, there is $j\in \{0,\ldots,k\}$ such that $\binom{l}{j}\in \Fqtimes$,
\item for $n+k\geq l>m$,   $\binom{l}{j}=0\in \Fq$ for  all $j\in \{0,\ldots,k\}$.
\end{enumerate}
As
$\binom{l}{j}=\binom{l-1}{j-1}+\binom{l-1}{j}$ for all $j\geq 1$, $l\geq 0$, we readily see that
if $\binom{l}{j}\neq 0\in \Fq$ for some $l$ and $j$, then $\binom{l-1}{j-1}\neq 0$ or $\binom{l-1}{j}\neq 0$, and hence there is some maximal $m$ for which a $j$ with $\binom{m}{j}\neq 0\in \Fq$ exists. By the possible ranges for $i$ and $j$, we obtain that $m$ indeed satisfies $n\leq m\leq n+k$.
\end{proof}

%

We recall from \cite[Prop.~12.7]{mr:ntff} that the Galois group of the  Carlitz $t^{m+1}$-torsion extension is isomorphic to $\left(\Fq[t]/(t^{m+1})\right)^\times$. As described  in \cite[Cor.~3.2]{am-rp:iddbcppte}, the Galois action of  $a\in \left(\Fq[t]/(t^{m+1})\right)^\times$ on the  $t^{m+1}$-torsion with respect to these generators can be given explicitly by 
\[ a* \left( \begin{matrix} \hd{m}{\omega}  \\  \hd{m-1}{\omega}  \\ \vdots \\ \omega  \end{matrix} \right)_{t=0} = \left( \begin{matrix} a   & \hd{1}{a}  & \cdots &\hd{m}{a}  \\ 0 & a  & \ddots & \vdots \\ \vdots & \ddots & \ddots & \hd{1}{a}   \\ 0 & \cdots & 0 & a  \end{matrix} \right)_{t=0}\cdot  \left( \begin{matrix} \hd{m}{\omega}  \\  \hd{m-1}{\omega}  \\ \vdots \\ \omega  \end{matrix} \right)_{t=0}
= \left( \begin{matrix} \hd{m}{a\omega}  \\  \hd{m-1}{a\omega}  \\ \vdots \\ a\omega  \end{matrix} \right)_{t=0}, \]
where the index $t=0$ means evaluation of all entries at $t=0$.
The induced Galois action on the $t$-adic Tate module of $C$ or respectively on $\hat{H}_C=\Fq\ps{t}\cdot \omega$ is then given as
\[  \Gal(K^\sep/K) \longsurjarrow \Fq\ps{t}^\times, \gamma\mapsto a:=a_\gamma \]
such that $\gamma(\omega)=a_\gamma\cdot \omega\in \CI\ps{t}$ where the action of $\gamma$ on $\omega$ is coefficient-wise.

\begin{thm}\label{thm:description-of-galois-rep}
The Galois representation on the $t$-adic Tate module of $\rho_kC$ is given by
\[ \begin{array}{rccccl}
 \varrho_t: &\Gal(K^{\sep}/K) &\longsurjarrow  &\Fq\ps{t}^\times & \longrightarrow &\GL_{k+1}(\Fq\ps{t}),\\ 
 & \gamma &\longmapsto & a:=a_\gamma & \longmapsto & \rho_{[k]}(a)=
\left( \begin{matrix} a   & \hd{1}{a}  & \cdots &\hd{k}{a}  \\ 0 & a  & \ddots & \vdots \\ \vdots & \ddots & \ddots & \hd{1}{a}   \\ 0 & \cdots & 0 & a  \end{matrix} \right).
\end{array} \]
Here the first map is the one given above.
\end{thm}

\begin{proof}
By Corollary \ref{cor:same-extensions}, the $t$-power torsion extensions $K(\rho_k C[t^\infty]) 
:=\bigcup_{n\geq 0} K(\rho_k C[t^{n+1}])$ and
$K(C[t^\infty]) :=\bigcup_{n\geq 0} K(C[t^{n+1}])$ of $K$ coincide. Hence, the representation factors through $\Fq\ps{t}^\times$.

Let $a=a_\gamma\in \Fq\ps{t}^\times$. As given above, its action on $\hd{l}{\omega}(0)$ is
$a*\hd{l}{\omega}(0)=\hd{l}{a\omega}(0)$. Hence the coefficient-wise action on $\hd{j}{\omega}$ is given as
\[ a*\hd{j}{\omega} = \sum_{i\geq 0} a*\left( \binom{i+j}{j}\hd{i+j}{\omega}(0)\right) t^i =  \sum_{i\geq 0} \binom{i+j}{j}\hd{i+j}{a\omega}(0) t^i
=  \hd{j}{a\omega}.
\]
Therefore, we obtain
\begin{eqnarray*}
\gamma \biggl( \begin{pmatrix} \omega & \hd{1}{\omega} & \cdots & \hd{k}{\omega} \\ 0 & \omega & \ddots & \vdots \\ \vdots & \ddots & \ddots & \hd{1}{\omega} \\ 0 & \cdots & 0 & \omega \end{pmatrix} \biggr) &=& \begin{pmatrix} a\omega & \hd{1}{a\omega} & \cdots & \hd{k}{a\omega} \\ 0 & a\omega & \ddots & \vdots \\ \vdots & \ddots & \ddots & \hd{1}{a\omega} \\ 0 & \cdots & 0 & a\omega \end{pmatrix}\\[2mm]
&=& \rho_{[k]}(a\omega)= \rho_{[k]}(a)\cdot \rho_{[k]}(\omega) \\[2mm]
&=&  \rho_{[k]}(a)\cdot \begin{pmatrix} \omega & \hd{1}{\omega} & \cdots & \hd{k}{\omega} \\ 0 & \omega & \ddots & \vdots \\ \vdots & \ddots & \ddots & \hd{1}{\omega} \\ 0 & \cdots & 0 & \omega \end{pmatrix},
\end{eqnarray*}
proving the claim.
\end{proof}

\begin{cor}\label{cor:image-Zariski-dense}
The image of $\varrho_t$ is Zariski-dense in $\Gamma_{\rho_k C}(\Fq\ls{t})$.
\end{cor}

\begin{proof}
We have to show that the ideal
\[ I =\{ P\in \Fq\ls{t}[X_0,\ldots, X_k] \mid \forall a\in \Fq\ps{t}^\times: P(a,\hd{1}{a},\ldots, \hd{k}{a})=0 \} \]
is the zero ideal.

Assume, for the contrary that $0\neq P\in I$, and let $i\leq k$ be the lowest index such that $X_i$ occurs in $P$, i.e.~$P\in \Fq\ls{t}[X_i,\ldots, X_k]$. We show that the coefficients of $P$ as a polynomial in $X_i$ over $\Fq\ls{t}[X_{i+1},\ldots, X_k]$ are also contained in $I$. Inductively, this shows that there is some non-zero $\Fq\ls{t$}-coefficient of $P$ which lies in $I$, a contradiction.

Choose $s\in \NN$ such that $q^s>k$. 
As for any $d\geq 1$ and $0<i<q^s$, the binomial coefficient $\binom{dq^s}{i}$ is divisible by the characteristic $p$, and hence equals $0$ modulo $p$, one sees that for all $b\in \Fq\ps{t^{q^s}}$, $\hd{i}{b}= 0$ for $0<i<q^s$.
Therefore using the generalized Leibniz rule \eqref{eq:leibniz-rule}, one obtains for all $b\in \Fq\ps{t^{q^s}}$, and all $j$ satisfying $i\leq j\leq k$:
\[   \hd{j}{t^ib}=\sum_{l=0}^j \hd{l}{t^i}\hd{j-l}{b}=\hd{j}{t^i}\cdot b =\partdef{  b & \text{for }j=i, \\ 0 & \text{for }i<j\leq k. } \]

For any $a\in \Fq\ps{t}^\times$, consider the polynomial $Q_a(Y)=P(\hd{i}{a}+Y,\hd{i+1}{a},\ldots, \hd{k}{a})\in \Fq\ls{t}[Y]$.
By the previous computation and assumption on $P$, we obtain for all $b\in \Fq\ps{t^{q^s}}$ (or in the case $i=0$, all  $b\in \Fq\ps{t^{q^s}}$ with $b(0)\neq -a(0)$)
\begin{eqnarray*}
Q_a(b) &=& P\left(\hd{i}{a}+b,\hd{i+1}{a},\ldots, \hd{k}{a}\right)\\
&=& P\left(\hd{i}{a}+\hd{i}{t^ib},\hd{i+1}{a}+\hd{i+1}{t^ib},\ldots, \hd{k}{a}+\hd{k}{t^ib}\right) \\
&=& P\left(\hd{i}{a+t^ib},\hd{i+1}{a+t^ib},\ldots, \hd{k}{a+t^ib}\right)\\
&=& 0
\end{eqnarray*}
This means that the polynomial $Q_a(Y)$ has infinitely many zeros, and therefore has to be identically to zero.
Hence, for all $a\in \Fq\ps{t}^\times$,
\[ 0=Q_a(X_i-\hd{i}{a})=P(X_i,\hd{i+1}{a},\ldots, \hd{k}{a}) \]
which implies that the coefficients of $P$ with respect to the variable $X_i$ also lie in $I$.
\end{proof}

\begin{cor}\label{cor:thin-image}
The image of the $t$-adic Galois representation $\varrho_t$ is thin with density $\delta(\varrho_t)=\frac{1}{k+1}$. In particular, it is not $t$-adically open in $\Gamma_{\rho_k C}(\Fq\ls{t})$.
\end{cor}

\begin{proof}
By Theorem \ref{thm:description-of-galois-rep} and Corollary \ref{cor:same-extensions}, the image of $\varrho_{t,N}:\Gal(K^\sep/K)\to \GL_{k+1}(\Fq\ps{t})\to  \GL_{k+1}(\Fq\ps{t}/(t^N))$ is in bijection to $\left(\Fq\ps{t}/t^{N+m}\right)^\times$ for some $0\leq m\leq k$ (which might depend on $N$). Hence for the order $D(N)$ of the image, we have
\[ (q-1)\cdot q^{N-1}\leq  D(N)=(q-1)\cdot q^{N+m-1}\leq (q-1)\cdot q^{N+k-1} \]
By  Corollary \ref{cor:image-Zariski-dense} and Theorem \ref{thm:description-of-motivic-grp}, we have $\dim(\overline{\im(\varrho_t)})=\dim(\Gamma_{\rho_k C})=k+1$, and hence
\[ \delta(\varrho_t)= \limsup_{N\to \infty} \frac{\log_q(D(N))}{N\cdot \dim (\Gamma_{\rho_k C})}
= \frac{1}{k+1}. \]
\end{proof}

We end this section by the computation of the density for the Carlitz tensor powers $E=C^{\otimes d}$.
By M.~Frantzen \cite[Thm.~5.2]{mf:novgram}, these have a non-open image if $d$ is divisible by the characteristic $p$ of $\Fq$.

\begin{thm}\label{thm:tensor-powers}
Let $d=p^e\cdot d'$ for $p=\ch(\Fq)$, $d'\in \NN$ prime to $p$, and $e\geq 0$. Let $E=C^{\otimes d}$ be the $d$-th tensor power of the Carlitz module, and $\varrho_{E}:\Gal(K^\sep/K)\to \GL_1(\Fq\ps{t})$ the associated $t$-adic Galois representation.

The density $\delta(\varrho_{E})$ equals $\frac{1}{p^e}$.
\end{thm}

\begin{proof}
As we explained earlier, the Galois action on the $t$-adic Tate module of the Carlitz module $C$ or respectively on $\hat{H}_C=\Fq\ps{t}\cdot \omega$ is given by a surjection
\[  \Gal(K^\sep/K) \longsurjarrow \Fq\ps{t}^\times, \gamma\mapsto a:=a_\gamma \]
such that $\gamma(\omega)=a_\gamma\cdot \omega\in \CI\ps{t}$ where the action of $\gamma$ on $\omega$ is coefficient-wise.
As the tensor power $E=C^{\otimes d}$ is by definition the $t$-module corresponding to the usual $d$-th tensor power of the $t$-motive
over $K[t]$, one has $\hat{H}_E=\Fq\ps{t}\cdot \omega^{d}$, and the $t$-adic Galois representation for $E$ is determined by the coefficient-wise action on $\omega^{d}$, i.e. it is the composition with the $d$-power map:
\[  \varrho_E:\Gal(K^\sep/K) \longsurjarrow \Fq\ps{t}^\times\to \Fq\ps{t}^\times, \gamma\mapsto a:=a_\gamma \mapsto a^d=(a^{d'})^{p^e}.\]
As $p=\ch(\Fq)$ and $d'$ is prime to $p$, the image of the $d$-power map is
\[ S=\{ f\in \Fq\ps{t^{p^e}}^\times \mid f(0)\text{ is a $d'$-power in }\Fqtimes. \}\] 
As in the definition of the density, let $D(N)$ be the size of the image of the composition of $\varrho_E$ with reduction modulo $t^N$, then
\[ D(N)=w\cdot q^{\lfloor \frac{N-1}{p^e}\rfloor}, \]
where $w$ denotes the number of $d'$-powers in $\Fqtimes$.
Hence,
\[ \delta(\varrho_E) = \limsup_{N\to \infty} \frac{\log_q(D(N))}{N\cdot 1 \cdot \dim(\GL_1)}
=\limsup_{N\to \infty} \frac{\log_q{w}+\lfloor \frac{N-1}{p^e}\rfloor}{N} = \frac{1}{p^e}. \qedhere \]

\end{proof}

\bibliographystyle{alpha}

\begin{thebibliography}{{Mau}18b}

\bibitem[And86]{ga:tm}
Greg~W. Anderson.
\newblock {$t$}-motives.
\newblock {\em Duke Math. J.}, 53(2):457--502, 1986.

\bibitem[AT90]{ga-dt:tpcmzv}
Greg~W. Anderson and Dinesh~S. Thakur.
\newblock Tensor powers of the {C}arlitz module and zeta values.
\newblock {\em Ann. of Math. (2)}, 132(1):159--191, 1990.

\bibitem[BP]{db-mp:ridmtt}
W.~Dale Brownawell and Matthew~A. Papanikolas.
\newblock A rapid introduction to {D}rinfeld modules, t-modules, and t-motives.
\newblock to appear in t-{M}otives: {H}odge {S}tructures, {T}ranscendence, and
  {O}ther {M}otivic {A}spects.

\bibitem[CP12]{cc-mp:aipldm}
Chieh-Yu Chang and Matthew~A. Papanikolas.
\newblock Algebraic independence of periods and logarithms of {D}rinfeld
  modules.
\newblock {\em J. Amer. Math. Soc.}, 25(1):123--150, 2012.
\newblock With an appendix by Brian Conrad.

\bibitem[Fra21]{mf:novgram}
Maike Frantzen.
\newblock Non-openness of v-adic {G}alois representation for {A}-motives.
\newblock {\em Int. J. Number Th.}, 17:33--53, 2021.

\bibitem[Mat89]{hm:crt}
Hideyuki Matsumura.
\newblock {\em Commutative ring theory}, volume~8 of {\em Cambridge Studies in
  Advanced Mathematics}.
\newblock Cambridge University Press, Cambridge, second edition, 1989.
\newblock Translated from the Japanese by M. Reid.

\bibitem[Mau18a]{am:ptmsv}
Andreas Maurischat.
\newblock Periods of t-modules as special values.
\newblock {\em {Journal of Number Theory}}, 2018.
\newblock {DOI}:10.1016/j.jnt.2018.09.024.

\bibitem[{Mau}18b]{am:ptmaip}
Andreas {Maurischat}.
\newblock Prolongations of t-motives and algebraic independence of periods.
\newblock {\em {Doc. Math.}}, 23:815--838, 2018.

\bibitem[MP19]{am-rp:iddbcppte}
Andreas Maurischat and Rudolph Perkins.
\newblock An integral digit derivative basis for {C}arlitz prime power torsion
  extensions.
\newblock {\em Publications Math\'ematiques de Besan\c con}, (1):131--149,
  2019.

\bibitem[Pap08]{mp:tdadmaicl}
Matthew~A. Papanikolas.
\newblock Tannakian duality for {A}nderson-{D}rinfeld motives and algebraic
  independence of {C}arlitz logarithms.
\newblock {\em Invent. Math.}, 171(1):123--174, 2008.

\bibitem[Pin97]{rp:mtcdm}
Richard Pink.
\newblock The {M}umford-{T}ate conjecture for {D}rinfeld-modules.
\newblock {\em Publ. Res. Inst. Math. Sci.}, 33(3):393--425, 1997.

\bibitem[Pin98]{rp:cslag}
Richard Pink.
\newblock Compact subgroups of linear algebraic groups.
\newblock {\em J. Algebra}, 206(2):438--504, 1998.

\bibitem[Ros02]{mr:ntff}
Michael Rosen.
\newblock {\em Number theory in function fields}, volume 210 of {\em Graduate
  Texts in Mathematics}.
\newblock Springer-Verlag, New York, 2002.

\end{thebibliography}
\def\cprime{$'$}

\vspace*{.5cm}

\parindent0cm

\end{document}